\documentclass[12pt]{article}
\usepackage{amsmath}
\usepackage{color}
\usepackage{fancyhdr}
\usepackage{verbatim}
\usepackage{amsfonts,longtable,mathrsfs}
\usepackage{epsfig}
\usepackage{amsfonts}
\usepackage{color, float}
\usepackage{fullpage}
\usepackage[numbers,sort&compress]{natbib}

\def\ec{\end{center}}
\def\bc{\begin{center}}
\def\ec{\end{center}}
\newtheorem{definition}{Definition}[section]
\newtheorem{theorem}{Theorem}[section]
\newtheorem{rem}{Remark}[section]

\newtheorem{proof}{Proof}
\begin{document}
\bc {\bf Group reduction, formulas solutions and asymptotic behavior of a class of fourth order difference equations
  }\ec
\medskip
\bc
Tshifhiwa Murovhi$^{1}$, Mensah Folly-Gbetoula $^{1}$\footnote{Corresponding author:\\ Mensah.Folly-Gbetoula@wits.ac.za(M. Folly-Gbetoula)\\
	ORCID: 0000-0002-3046-0679}, Anani Kwassi $^{2}$
\vspace{1cm}
\\$^{1}$School of Mathematics, University of the Witwatersrand, Wits 2050, Johannesburg,\\ South Africa.\\
$^{2}$ Department of Mathematics, University of Lomé,
{Lomé},
	{01 BP 1515}, {Togo}

\ec
\begin{abstract}
\noindent
The symmetry method is a powerful and systematic approach for solving difference equations. It leverages the concept of transformations that leave a difference equation invariant, simplifying its structure and often reducing the equation to a solvable form. In this paper, the symmetry method is employed to study some class of difference equations.  Using analytical techniques and computational tools, we derive explicit solutions for these equations and establish conditions for the existence of periodic solutions. Stability analysis is performed to identify non-hyperbolic  points. Furthermore, some asymptotic properties of the  difference equations are explored, with results and graphs illustrating how initial conditions and parameter values influence the behavior of the solutions.
\end{abstract}
\section{Introduction and background} \setcounter{equation}{0}
Difference equations are a fundamental mathematical tool used to describe discrete values of a variable across intervals, often representing time, space, or other measurable dimensions. They are the discrete counterpart of differential equations and are extensively applied in fields such as mathematics, physics, economics, biology, and computer science, where systems naturally evolve in steps rather than continuously. 
Difference equations form the basis of discrete dynamical systems, enabling the study of stability, periodicity, and chaotic behavior in models. 
\par \noindent The study of difference equations often involves exploring their solutions' stability, convergence, periodic nature, long-term behavior and analytical solutions. However, there is still a lot to be done when it comes to deriving exact solutions.  There are several methods developed in the study of recurrence equations. In this paper, we use the well-known Lie symmetry approach developed by S. Lie  for differential equations. Its application to their discrete counterpart is somewhat new. To the best of our knowledge, S. Maeda (1987) was the first to adapt Lie methods to difference equations by developing similar techniques that allowed one to reduce the number of variables, linearize and even derive analytical solutions for difference equations. Maeda's work laid the groundwork for subsequent advancements in the application of symmetries to difference equations, enhancing their theoretical and practical analysis. Without taking out anything from the work of Maeda, it is worthwhile mentioning that the recent interest in the application of this method to difference equations is rekindled by the work of Peter Hydon in his book \cite{Hydon}. Hydon provided a systematic method for finding symmetries for difference equations. Although his algorithm works for difference equations of any order, he mainly applied it to lower order equations. The application of Lie analysis to higher order difference equations is one of the key contributions of Folly-Gbetoula and others \cite{M1, M2, M3, M4}. 
\par \noindent In this paper, we investigate the existence of solutions  of the class of fourth order difference equations of the form
\begin{align}\label{eq1}
x_{n+4}=\frac{x_{n}x_{n+3}}{a_nx_{n+1}+ b_n x_{n}x_{n+2}},
\end{align}
for some arbitrary real sequences $a_n$ and $b_n$, by deriving the vectors field that spans its Lie algebra. We then utilize the formulas for the solutions to better understand their periodic nature and asymptotic behavior. For various approaches to the study of difference equations, see \cite{N1,N2}.
\par 
\noindent 
\subsection{Background work}
As mentioned earlier, Hydon did a lot of work when it comes to the extension of Lie group analysis to difference equations. This involved generalizing continuous symmetry techniques to discrete settings. This method enables the study of integrability, reduction of complexity and formula solutions for difference equations. Most of our notation follows the ones adopted by Hydon in \cite{Hydon}. Interested readers can refer to the books \cite{Hydon, Grove} from which the following definitions and theorems are taken.
\par 
\noindent 
Consider the fourth-order difference equation of independent variable $n$ and dependent variables $x_n$ and its shifts
\begin{align}\label{eq2}
x_{n+4}=\Omega(n,x_n,x_{n+1},x_{n+2},x_{n+3}).
\end{align}
\begin{definition}
The forward shift operator is defined by 
\begin{align}
S^{(j)}:n\rightarrow n+j.
\end{align}
\end{definition}
We consider the one parameter Lie group of point transformations $\Gamma_\varepsilon$
\begin{align}\label{gtrans}
 \hat{n}=&n,\nonumber\\ \hat{x}_{n}=&x_n+ \varepsilon Q(n,x_{n})+O(\varepsilon ^2),
\end{align}
where $Q$ is known as the characteristic function and $\varepsilon$ is the parameter of the group of point  transformations. Symmetries are often represented by generators that describe infinitesimal transformations. In this paper, we define the infinitesimal prolonged generator corresponding to \eqref{gtrans} by 
 \begin{align}\label{vectorp}
 \textbf{X}^{[3]}=&Q(n,x_n)\frac{\partial\quad}{\partial x_n}+Q(n+1,x_{n+1})\frac{\partial\quad}{\partial x_{n+1}}+Q(n+2,x_{n+2})\frac{\partial\quad}{\partial x_{n+2}}\nonumber\\&+Q(n+3,x_{n+3})\frac{\partial\quad}{\partial x_{n+3}}.
 \end{align}
We refer to 
 \begin{align}\label{vector}
 	\textbf{X}=Q(n,x_n)\frac{\partial\quad}{\partial x_n}
 \end{align}
 as the infinitesimal generator.
\begin{definition}
A group of transformations $\Gamma_{\varepsilon}$ is a symmetry of \eqref{eq2} if and only if 
\begin{align}\label{symc}
S^4Q(n, u_n)-\textbf{X}^{[3]}\Omega =0
\end{align}
whenever \eqref{eq2} is true.
\end{definition} 
There are functions that remain unchanged under the symmetry transformations. They are called invariants and are often used to simplify the difference equation. \\
The symmetry condition \eqref{symc} often yields a system of equations that is solvable even though it usually involves cumbersome calculations when dealing with higher order equations. 
\par \noindent The main role of these symmetries is reduction of order or complexity of the equation. In the sequel, we  will require the knowledge of the canonical coordinate \cite{NV}:
\begin{align}
S_n=\int \frac{dx_{n}}{Q(n,x_{n})}.
\end{align}
\par \noindent 
\section{Symmetries, reduction and analytical solutions}
Consider the fourth order difference equation 
\begin{align}\label{xn}
x_{n+4}=\frac{x_{n}x_{n+3}}{a_nx_{n+1}+ b_n x_{n}x_{n+2}},
\end{align}
where $a_n$ and $b_n$ are some random sequences. Imposing the criterion of invariance \eqref{symc} on \eqref{xn} yields
\begin{align}\label{a1}
 &Q(n+4,\Omega)- \textbf{X}^{[3]}\Omega=0,
\end{align}
whenever \eqref{xn} holds. Note that $\Omega$ is the right-hand side expression in \eqref{xn}. The procedure for solving this type of functional equation consists of lengthy computations. 
\begin{itemize}
\item Firstly, we apply the differential operator $$L=
\frac{\partial }{\partial x_n}-\frac{a_nx_{n+1}x_{n+3}}{b_nx_n^2x_{n+2}+a_nx_nx_{n+1}}\frac{\partial }{\partial x_{n+3}}$$ on \eqref{a1} to get
\begin{align}
	&\frac{a_n x_{n+1}x_{n+3} }{(b_nx_nx_{n+2}+a_nx_{n+1})^2} Q'(n+3,x_{n+3})  - \frac{a_n x_{n+1} }{(b_nx_nx_{n+2}+a_nx_{n+1})^2} Q(n+3,x_{n+3})  \nonumber\\
	& + \frac{a_nb_n x_n x_{n+1}x_{n+3} }{(b_nx_nx_{n+2}+a_nx_{n+1})^3} Q(n+2,x_{n+2}) -\frac{a_nb_n x_n x_{n+2}x_{n+3} }{(b_nx_nx_{n+2}+a_nx_{n+1})^3} Q(n+1,x_{n+1})\nonumber
	\\
	&- \frac{a_n x_n x_{n+1}x_{n+3} }{(b_nx_nx_{n+2}+a_nx_{n+1})^2} Q'(n,x_{n}) +  \frac{a_n x_{n+1}x_{n+3}(2b_nx_nx_{n+2}+a_nx_{n+1}) }{x_n(b_nx_nx_{n+2}+a_nx_{n+1})^3} Q(n,x_{n}).
\end{align}
\item Secondly, we multiply through by $(a_nx_{n+1}+b_nx_nx_{n+2})^3$ and then differentiate the resulting equation with respect to $x_n$ twice to get (after simplification)
\begin{align}
&-(b_nx_nx_{n+2}+a_nx_{n+1})Q'''\left(n, x_n\right) + \frac{a_n x_{n+1} }{x_n}Q''\left(n, x_n\right) - \frac{2 \, a_n x_{n+1}}{x_n^{2}}Q'\left(n, x_n\right) \nonumber\\
&+ \frac{2 \, a_n x_{n+1} }{x_n^{3}}Q\left(n, x_n\right).
\end{align}
\item Thirdly, we apply the method of separation to obtain the following overdetermined system of differential equations: 
\begin{align}
	x_{n+1} &:& -a_nQ'''\left(n, x_n\right)+\frac{a_n  }{x_n}Q''\left(n, x_n\right)- \frac{2 \, a_n }{x_n^{2}}Q'\left(n, x_n\right) +\frac{2 \, a_n }{x_n^{3}}Q\left(n, x_n\right)  =0  \\
	x_{n+2} &:& -b_nx_n Q'''\left(n, x_n\right)=0.
\end{align}

\item Finally, we solve for $Q$ to get $Q(n, x_n)= 	\alpha _n  x_n + \beta_n x_n^2 + \gamma_n$. Then, we substitute this expression in \eqref{a1} to eliminate all dependency among the arbitrary constants. This yields $\beta_n =0$, $\gamma_n=0$ and 
\begin{align}\label{cons}
	\alpha _n-\alpha_{n+1}+ \alpha_{n+2}=0.
\end{align}
\end{itemize}
We obtain that
\begin{align}\label{charac}
Q(n,u_n)= \alpha_n  x_n,
\end{align}
where $\alpha _n$ satisfies \eqref{cons}. That is to say  $\alpha_n=\exp{(\pm i n\pi/3)}$.
The Lie algebra of the equation under study is then spanned by 
\begin{align}\label{gener}
&X_{1}=\exp{(- i n\pi/3)}x_n{\partial { x_{n}}} \quad \text{and} \quad X_{2}=\exp{( i n\pi/3)} x_{n}{\partial { x_{n}}}.
\end{align}
We let the associated canonical coordinate be
\begin{equation}\label{a7}
\tilde{V}_n= \int\frac{dx_n}{\alpha_n x_{n}} \end{equation}
and we choose the invariant (it is to check that $X_1^{[3]}( V_n )=X_2^{[3]}( V_n )=0$)
\begin{equation}  V_n=1/\exp(\alpha_n\tilde{V}_n- \alpha_{n+1}\tilde{V}_{n+1} +\alpha_{n+2}\tilde{V}_{n+2}).
\end{equation}
Clearly,
\begin{align}
V_n=\frac{x_{n+1}}{x_nx_{n+2}}.
\end{align} 
It follows that
\begin{equation}\label{a9}
V_{n+2}={a}_n V_n+ {b}_n \quad \text{and} \quad 	x_{n+6}= \frac{V_{n}V_{n+1}}{V_{n+3}V_{n+4}}x_n.
\end{equation}
Straightforward iterations of  equations in \eqref{a9} yield
\begin{align}
V_{2n+j}=&V_j\left(   \prod _{k_1=0}^{n-1}{a}_{2k_1+j}\right) +\sum _{l=0}^{n-1} \left( {b}_{2l+j}\prod _{k_2=l+1}^{n-1}{a}_{2k_2+j}\right),\quad j=0, 1,\label{a10a}\\
x_{6n+i}=& x_i \left(   \prod _{k_1=0}^{n-1}  \frac{V_{6k_1+i}V_{6k_1+i+1}}{V_{6k_1+i+3}V_{6k_1+i+4}}\right),\quad j=0, 1, 2 ,3 ,4 ,5.\label{a10b}
\end{align}
Noting that every integer can be written as $n=2\lfloor \frac{n}{2}\rfloor +\tau{(n)}$, where $\lfloor\cdot\rfloor$ denotes the floor function and $\tau{(n)}$ is the remainder when $n$ is divided by $2$, equation \eqref{a10b} can take the form
\begin{align}\label{a11}
	x_{6n+i}=& x_i \left(   \prod _{k_1=0}^{n-1}  \frac{V_{2(3k_1+\lfloor \frac{i}{2}\rfloor)+\tau(i)}V_{2(3k_1+\lfloor \frac{i+1}{2}\rfloor)+\tau(i+1)}}{V_{2(3k_1+1+\lfloor \frac{i+1}{2}\rfloor)+\tau(i+1)}V_{2(3k_1+2+\lfloor \frac{i}{2}\rfloor)+\tau(i)}}\right).
\end{align}
Employing \eqref{a10a} in \eqref{a11}, we have that
\begin{align}\label{a12}
	&x_{6n+i}= x_i  \prod _{k_1=0}^{n-1} \left( \frac{V_{\tau(i)}\left(   \prod\limits _{k_0=0}^{3k_1+\lfloor \frac{i}{2}\rfloor)-1}{a}_{2k_0+\tau(i)}\right) +\sum\limits _{l=0}^{3k_1+\lfloor \frac{i}{2}\rfloor-1} \left( {b}_{2l+\tau(i)}\prod\limits _{k_2=l+1}^{3k_1+\lfloor \frac{i}{2}\rfloor-1}{a}_{2k_2+\tau(i)}\right)}{V_{\tau(i+1)}\left(   \prod\limits _{k_0=0}^{3k_1+\lfloor \frac{i+1}{2}\rfloor}{a}_{2k_0+\tau(i+1)}\right) +\sum\limits _{l=0}^{3k_1+\lfloor \frac{i+1}{2}\rfloor} \left( {b}_{2l+\tau(i+1)}\prod\limits _{k_2=l+1}^{3k_1+\lfloor \frac{i+1}{2}\rfloor}{a}_{2k_2+\tau(i+1)}\right)}\right)\nonumber\\
	& \times  \left( \frac{V_{\tau(i+1)}\left(   \prod\limits _{k_0=0}^{3k_1+\lfloor \frac{i+1}{2}\rfloor-1}{a}_{2k_0+\tau(i+1)}\right) +\sum\limits _{l=0}^{3k_1+\lfloor \frac{i+1}{2}\rfloor-1} \left( {b}_{2l+\tau(i+1)}\prod\limits _{k_2=l+1}^{3k_1+\lfloor \frac{i+1}{2}\rfloor-1}{a}_{2k_2+\tau(i+1)}\right)}{V_{\tau(i)}\left(   \prod\limits _{k_0=0}^{3k_1+\lfloor \frac{i}{2}\rfloor)+1}{a}_{2k_0+\tau(i)}\right) +\sum\limits _{l=0}^{3k_1+\lfloor \frac{i}{2}\rfloor+1} \left( {b}_{2l+\tau(i)}\prod\limits _{k_2=l+1}^{3k_1+\lfloor \frac{i}{2}\rfloor+1}{a}_{2k_2+\tau(i)}\right)}\right),
\end{align}
for $i=0,1,2,3,4,5$. The above equation gives the closed form solution of the equation under investigation.
\begin{rem}
We remark that equation \eqref{a10b} can be written in a unified manner. In fact,
\begin{align}\label{group}
	x_{6n+j}=&x_j \prod_{k=0}^{n-1}\frac{V_{6k+j}V_{6k+j+1}} {V_{6k+j+3}V_{6k+j+4}}\\	
	=&H_{j}\exp\left[-\frac{2}{\sqrt{3}  }\sum_{k=0}^{6n+j-1} \sin\left(\frac{j-k-1}{3}\pi\right)\ln V_k    \right] \nonumber\\
	=&	H_{6n+j}\exp\left[-\frac{2}{\sqrt{3}  }\sum_{k=0}^{6n+j-1} \sin\left(\frac{6n+j-k-1}{3}\pi\right)\ln V_k    \right],
\end{align}
where the $V_k$'s are given in \eqref{a9} and $H_n$'s are such that 
\begin{align}
	H_0=x_0,\; H_1=x_1,\; H_2=\frac{x_1}{x_0},\; H_3=\frac{1}{x_0},\;H_4=\frac{1}{x_1},\; H_5=\frac{x_0}{x_1} \quad \text{and} \quad H_{6n+j}=H_j,
\end{align}
for  $j=0,1,2,3,4,5$. It follows that 
\begin{align}\label{uni}
	x_n=H_n\exp\left[-\frac{2}{\sqrt{3}  }\sum_{k=0}^{n-1} \sin\left(\frac{n-k-1}{3}\pi\right)\ln V_k    \right].
\end{align}
\end{rem}
After analyzing the general case, it is essential to explore specific instances that illustrate the theory's application and nuances. By examining these special cases, we can gain deeper insights into the behavior of the solutions under particular conditions. 
\section{The case where $a_n$ and $b_n$ are $1$-periodic sequences}
Suppose $(a_n)_{n\geq0} =(a,a,\dots)$ and $(b_n)_{n\geq0} =(b,b,\dots)$. In this instance, \eqref{a12} simplifies into 
\begin{align}\label{abcst}
	x_{6n+i}=& x_i  \prod _{k_1=0}^{n-1} \left( \frac{V_{\tau(i)}a^{3k_1+\lfloor \frac{i}{2}\rfloor} +b\sum\limits _{l=0}^{3k_1+\lfloor \frac{i}{2}\rfloor-1}a^l}{V_{\tau(i+1)}a^{3k_1+\lfloor \frac{i+1}{2}\rfloor+1} +b\sum\limits _{l=0}^{3k_1+\lfloor \frac{i+1}{2}\rfloor}a^l}\right) \left( \frac{V_{\tau(i+1)}a^{3k_1+\lfloor \frac{i+1}{2}\rfloor} +b\sum\limits _{l=0}^{3k_1+\lfloor \frac{i+1}{2}\rfloor-1}a^l}{V_{\tau(i)}a^{3k_1+\lfloor \frac{i}{2}\rfloor+2} +b\sum\limits _{l=0}^{3k_1+\lfloor \frac{i}{2}\rfloor+1} a^l}\right),\nonumber\\
	=& x_i  \prod _{k_1=0}^{n-1} \left( \frac{a^{3k_1+\lfloor \frac{i}{2}\rfloor} +\frac{b}{V_{\tau(i)}}\sum\limits _{l=0}^{3k_1+\lfloor \frac{i}{2}\rfloor-1}a^l}{a^{3k_1+\lfloor \frac{i+1}{2}\rfloor+1} +\frac{b}{V_{\tau(i+1)}}\sum\limits _{l=0}^{3k_1+\lfloor \frac{i+1}{2}\rfloor}a^l}\right) \left( \frac{a^{3k_1+\lfloor \frac{i+1}{2}\rfloor} +\frac{b}{V_{\tau(i+1)}}\sum\limits _{l=0}^{3k_1+\lfloor \frac{i+1}{2}\rfloor-1}a^l}{a^{3k_1+\lfloor \frac{i}{2}\rfloor+2} +\frac{b}{V_{\tau(i)}}\sum\limits _{l=0}^{3k_1+\lfloor \frac{i}{2}\rfloor+1} a^l}\right)
\end{align}
for $i=0,1,2,3,4,5$, where $V_i=x_{i+1}/(x_ix_{i+2})$.
\subsection{The case where $a=1$}
We note that \eqref{abcst} becomes
\begin{align}\label{a1bcst}
	x_{6n+i}	=& x_i  \prod _{k_1=0}^{n-1} \left( \frac{1 +\frac{bx_{\tau(i)}x_{\tau(i)+2}}{x_{\tau(i)+1}}(3k_1+\lfloor \frac{i}{2}\rfloor)}{1 +\frac{bx_{\tau(i+1)}x_{\tau(i+1)+2}}{x_{\tau(i+1)+1}}(3k_1+\lfloor \frac{i+1}{2}\rfloor+1)}\right) \left( \frac{1 +\frac{bx_{\tau(i+1)}x_{\tau(i+1)+2}}{x_{\tau(i+1)+1}}(3k_1+\lfloor \frac{i+1}{2}\rfloor)}{1 +\frac{bx_{\tau(i)}x_{\tau(i)+2}}{x_{\tau(i)+1}}(3k_1+\lfloor \frac{i}{2}\rfloor+2)}\right)
\end{align}
when $a=1$. More explicitly, we have 
\begin{align}\label{a1bcst'}
	x_{6n}	=& x_0  \prod _{k_1=0}^{n-1} \left( \frac{1 +\frac{bx_{0}x_{2}}{x_{1}}(3k_1)}{1 +\frac{bx_{1}x_{3}}{x_{2}}(3k_1+1)}\right) \left( \frac{1 +\frac{bx_{1}x_{3}}{x_{2}}(3k_1)}{1 +\frac{bx_{0}x_{2}}{x_{1}}(3k_1+2)}\right),\nonumber\\
	x_{6n+1}	=& x_1  \prod _{k_1=0}^{n-1} \left( \frac{1 +\frac{bx_{1}x_{3}}{x_{2}}(3k_1)}{1 +\frac{bx_{0}x_{2}}{x_{1}}(3k_1+2)}\right) \left( \frac{1 +\frac{bx_{0}x_{2}}{x_{1}}(3k_1+1)}{1 +\frac{bx_{1}x_{3}}{x_{2}}(3k_1+2)}\right),\nonumber\\
	x_{6n+2}	=& x_2  \prod _{k_1=0}^{n-1} \left( \frac{1 +\frac{bx_{0}x_{2}}{x_{1}}(3k_1+1)}{1 +\frac{bx_{1}x_{3}}{x_{2}}(3k_1+2)}\right) \left( \frac{1 +\frac{bx_{1}x_{3}}{x_{2}}(3k_1+1)}{1 +\frac{bx_{0}x_{2}}{x_{1}}(3k_1+3)}\right),
	\nonumber\\
	x_{6n+3}	=& x_3  \prod _{k_1=0}^{n-1} \left( \frac{1 +\frac{bx_{1}x_{3}}{x_{2}}(3k_1+1)}{1 +\frac{bx_{0}x_{2}}{x_{1}}(3k_1+3)}\right) \left( \frac{1 +\frac{bx_{0}x_{2}}{x_{1}}(3k_1+2)}{1 +\frac{bx_{1}x_{3}}{x_{2}}(3k_1+3)}\right),
	\nonumber\\
	x_{6n+4}	=& \frac{x_{0}x_{3}}{x_{1}+ b x_{0}x_{2}} \prod _{k_1=0}^{n-1} \left( \frac{1 +\frac{bx_{0}x_{2}}{x_{1}}(3k_1+2)}{1 +\frac{bx_{1}x_{3}}{x_{2}}(3k_1+3)}\right) \left( \frac{1 +\frac{bx_{1}x_{3}}{x_{2}}(3k_1+2)}{1 +\frac{bx_{0}x_{2}}{x_{1}}(3k_1+4)}\right),
	\nonumber\\
	x_{6n+5}	=& \frac{x_{0}x_1x_{3}}{(x_{1}+ b x_{0}x_{2})(x_2+bx_1x_3)}  \prod _{k_1=0}^{n-1} \left( \frac{1 +\frac{bx_{1}x_{3}}{x_{2}}(3k_1+2)}{1 +\frac{bx_{0}x_{2}}{x_{1}}(3k_1+4)}\right) \left( \frac{1 +\frac{bx_{0}x_{2}}{x_{1}}(3k_1+3)}{1 +\frac{bx_{1}x_{3}}{x_{2}}(3k_1+4)}\right).
\end{align}
\subsection{The case where $a\neq1$}
Here, \eqref{abcst} becomes
\begin{align}\label{anot1}
	x_{6n+i}		=& x_i  \prod _{k_1=0}^{n-1} \left( \frac{a^{3k_1+\lfloor \frac{i}{2}\rfloor} +\frac{b}{V_{\tau(i)}}\left(\frac{1-a^{3k_1+\lfloor \frac{i}{2}\rfloor}}{1-a}\right)}{a^{3k_1+\lfloor \frac{i+1}{2}\rfloor+1} +\frac{b}{V_{\tau(i+1)}}\left(\frac{1-a^{3k_1+\lfloor \frac{i+1}{2}\rfloor+1}}{1-a}\right)}\right) \left( \frac{a^{3k_1+\lfloor \frac{i+1}{2}\rfloor} +\frac{b}{V_{\tau(i+1)}}\left(\frac{1-a^{3k_1+\lfloor \frac{i+1}{2}\rfloor}}{1-a}\right)}{a^{3k_1+\lfloor \frac{i}{2}\rfloor+2} +\frac{b}{V_{\tau(i)}}\left(\frac{1-a^{3k_1+\lfloor \frac{i}{2}\rfloor+2}}{1-a}\right)}\right).
\end{align}
The case $a=-1$ simplifies considerably into  
\begin{align}\label{a-1b}
	x_{6n+i}	=& \begin{cases}
		x_i  \quad \text {if} \quad n \quad \text{even}\\
		 x_i \left(-1+\frac{b}{V_{\tau(i+1)}} \right) ^{(-1)^{\lfloor\frac{i+1}{2}\rfloor+1}}\quad \text {if} \quad n \quad \text{odd}.
	\end{cases}
\end{align}
Consequently, this result establishes the following:
\begin{align}\label{12}
	x_{12n+i}=x_i,
\end{align}
$i=0,1,2,3,4,5$ and for all $n$ when $a=-1$.
\section{Periodicity, stability and long-term behavior}
The focused study in the previous section pave way to periodicity, stability and other behaviors. We start by looking at the periodic nature of the solutions. The following theorems establish the conditions under which the solutions exhibit periodicity and characterize their  periods.
\begin{theorem}\label{1}
Consider the equation
	{\footnotesize 
		\begin{equation}\label{anotpm1}
			x_{n+4}=\frac{x_{n}x_{n+3}}{ax_{n+1}+bx_{n}x_{n+2}},
	\end{equation}}
where $|a|\neq 1$ and $b$ some non-zero constants; and $x_i, \; i={0, 1, 2, 3}$ are the initial conditions. Assume the initial conditions   are such that 
	\begin{align}\label{cdanot1}
		&\frac{x_{0}x_{2}}{x_{1}}=\frac{x_{1}x_{3}}{x_{2}}=\frac{1-a}{b}
	\quad \textit{and}
	\quad
	 x_i\neq x_{i+2}, x_i\neq x_{i+3}.	\end{align}
	Then every solution $(x_n)_{n\geq0}$ of \eqref{anotpm1} satisfying \eqref{cdanot1} is periodic with period six.
\end{theorem}
\begin{proof} From the first condition in \eqref{cdanot1}, the following can be inferred: $V_{\tau(i)}=V_{\tau(i+1)}=\frac{b}{1-a}$. Using this assumption in \eqref{anot1}, we have that  
	\begin{align}
		x_{6n+i}		=& x_i  \prod _{k_1=0}^{n-1} \left( \frac{a^{3k_1+\lfloor \frac{i}{2}\rfloor} +\frac{b(1-a)}{b}\left(\frac{1-a^{3k_1+\lfloor \frac{i}{2}\rfloor}}{1-a}\right)}{a^{3k_1+\lfloor \frac{i+1}{2}\rfloor+1} +\frac{b(1-a)}{b}\left(\frac{1-a^{3k_1+\lfloor \frac{i+1}{2}\rfloor+1}}{1-a}\right)}\times  \frac{a^{3k_1+\lfloor \frac{i+1}{2}\rfloor} +\frac{b(1-a)}{b}\left(\frac{1-a^{3k_1+\lfloor \frac{i+1}{2}\rfloor}}{1-a}\right)}{a^{3k_1+\lfloor \frac{i}{2}\rfloor+2} +\frac{b(1-a)}{b}\left(\frac{1-a^{3k_1+\lfloor \frac{i}{2}\rfloor+2}}{1-a}\right)}\right)\nonumber \\=&x_j.\nonumber
	\end{align}
Thus, the period of the solution divides six. With the second condition in \eqref{cdanot1}, it follows that we have a periodic solution with period six. 
\end{proof}
The graphs below provide graphical representations of the result established in Theorem \ref{1}. They illustrate how the theorem's conditions and conclusion manifest confirming the theoretical finding. The initial conditions used in Figure \ref{DM21_1} satisfy Condition \eqref{cdanot1}. We see six periodic solutions as predicted. 
\begin{figure}[H]
		\centering
		\includegraphics[scale=0.3]{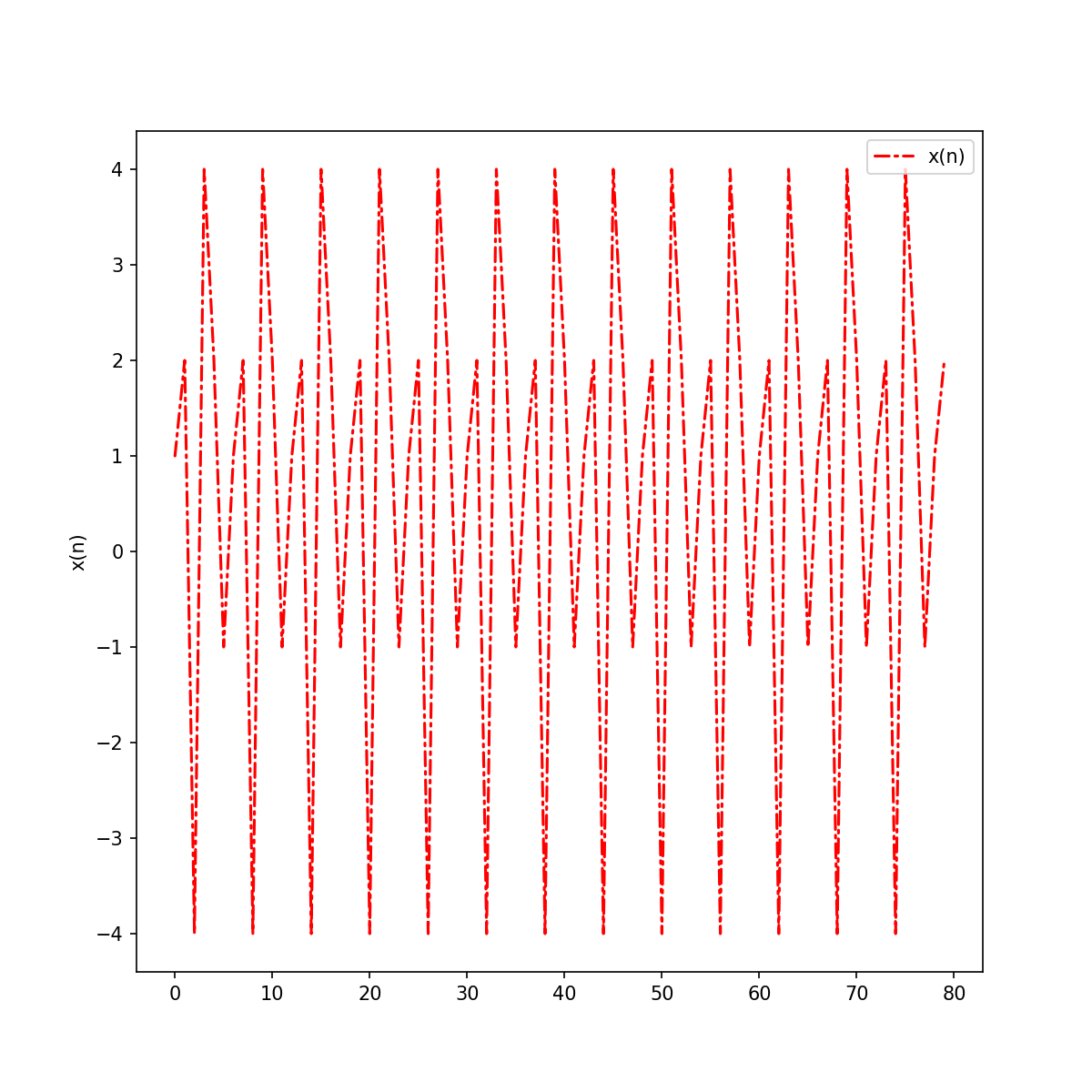}
		\caption{Graph of \eqref{anotpm1} when $a=3, b=1, x_0=1, x_1=2, x_2=-4, x_3=4$.\label{DM21_1}}
		\centering
\end{figure}
\begin{theorem}\label{2}
	Consider the equation
	{\footnotesize 
		\begin{equation}\label{a-1}
			x_{n+4}=\frac{x_{n}x_{n+3}}{-x_{n+1}+bx_{n}x_{n+2}}
	\end{equation}}
	for some constant $b\neq0$; and $x_i, \; i=\overline{0, 3}$ are the initial conditions. 
	Then, every solution $(x_n)_{n>0}$ of \eqref{a-1} is periodic with period twelve.
\end{theorem}
\begin{proof} The result follows from \eqref{a-1b} and \eqref{12}. In fact, the twelve periodic solutions are the following: 
\begin{align}
&\dots,	 x_0, x_1, x_2, x_3, \frac{x_0x_3}{-x_1+bx_0x_2}, \frac{x_0x_1x_3}{(-x_1+bx_0x_2)(-x_2+bx_1x_3)}, x_0\left(-1+\frac{bx_1x_3}{x_2}\right) ^{-1}, \nonumber \\& x_1\left(-1+\frac{bx_0x_2}{x_1}\right),
x_2\left(-1+\frac{bx_1x_3}{x_2}\right),
x_3\left(-1+\frac{bx_0x_2}{x_1}\right) ^{-1},
\frac{x_0x_3}{-x_1+bx_0x_2}\left(-1+\frac{bx_1x_3}{x_2}\right) ^{-1},\nonumber\\
& \frac{x_0x_1x_3}{(-x_1+bx_0x_2)(-x_2+bx_1x_3)}\left(-1+\frac{bx_0x_2}{x_1}\right), \dots.
\end{align}
\end{proof}
Similarly, Figure \ref{DM21_2} shows a graphical representation of the result established in Theorem \ref{2}. We observe twelve periodic solutions as predicted. 
\begin{figure}[H]
		\centering
				\includegraphics[scale=0.3]{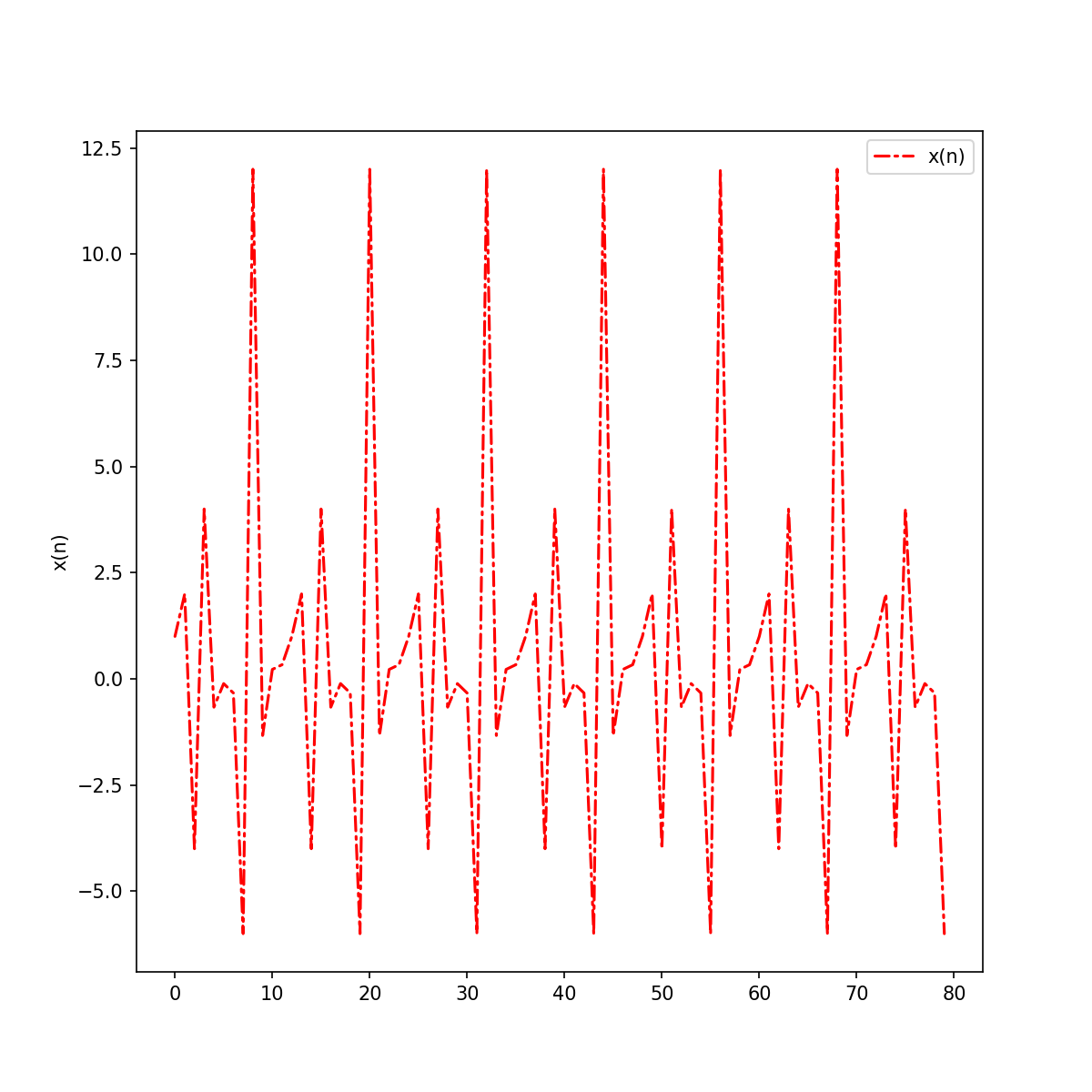}
		\caption{Graph of \eqref{a-1} when $a=-1, b=1, x_0=1, x_1=2, x_2=-4, x_3=4$.\label{DM21_2}}
		\centering
\end{figure}
\begin{theorem}
Given Equation \eqref{anotpm1} with $a\neq 1$. The unique equilibrium point ${x}=(1-a)/b$ is non-hyperbolic.
\end{theorem}
\begin{proof}
The equilibrium point is obtained by solving the equation $x=x^2/(ax+bx^2)$. Solving this equation for $x$, we found that $x=(1-a)/b$. The characteristic equation of  \eqref{anotpm1} near $0$ is therefore given by 
\begin{align}\label{charanot1}
	\lambda ^{4}-\lambda ^{3}-(a-1)\lambda ^{2}+ a\lambda-a=0,
	\end{align}
that is,	
\begin{align}
	(\lambda ^{2}-\lambda +1)( \lambda ^{2}-{a})=0.
\end{align}
The equation $\lambda ^2 -\lambda +1=0$ has two roots whose moduli are all equal to 1. Consequently, $x=(1-a)/b$ is  non-hyperbolic when $a\neq 1$.
\end{proof}
\section{Conclusion}
In this study, we explored the Lie symmetries, analytical solutions, periodicity, and stability of a family of fourth order difference equations, providing a analysis of their dynamic behavior. By employing the Lie symmetry method, we identified symmetries that facilitated the reduction and simplification of the equation under study, enabling the derivation of explicit analytical solutions. These solutions not only validate the underlying theoretical framework but also highlight the applicability of symmetry-based approaches in solving discrete dynamical systems. The investigation into periodicity revealed conditions under which solutions exhibit repetitive behavior, offering insights into the long-term dynamics of these systems. Conditions of existence of 6-, 12- periodic solutions and  non-hyperbolic equilibrium points were clearly stated.


\end{document}